\documentclass{scrartcl}
\usepackage{amsmath,amssymb,amsfonts,amsthm} 
\usepackage{graphics,graphicx}                 
\usepackage{color}                    
\usepackage{hyperref,fancyhdr}                 
\usepackage[all]{xy}
\usepackage[applemac]{inputenc}
\usepackage[ngerman,english]{babel}
\usepackage{url}
\usepackage{abstract}

\newtheorem{theorem}{Theorem}[section]
\newtheorem{proposition}[theorem]{Proposition}
\newtheorem{corollary}[theorem]{Corollary}
\newtheorem{lemma}[theorem]{Lemma}

\newtheorem{theorem-example}[theorem]{Theorem $\backslash$ Example}
\renewenvironment{proof}{\begin{sloppypar}\noindent{\bf Proof}}{\hfill$\blacksquare$\end{sloppypar}}
\theoremstyle{definition}

\newtheorem{remark}[theorem]{Remark}

\newtheorem{open problem}[theorem]{Open Problem}
\DeclareMathOperator{\Aut}{Aut}

\DeclareMathOperator{\Hom}{Hom}

\begin{document}

\title{\bf{Extending Characters of\\Fixed Point Algebras}}

\author{Stefan Wagner}

\maketitle

\begin{abstract}
\noindent
A dynamical system is a triple $(A,G,\alpha)$, consisting of a unital locally convex algebra $A$, a topological group $G$ and a group homomorphism $\alpha:G\rightarrow\Aut(A)$, which induces a continuous action of $G$ on $A$. Furthermore, a unital locally convex algebra $A$ is called a continuous inverse algebra, or CIA for short, if its group of units $A^{\times}$ is open in $A$ and the inversion map \mbox{$\iota:A^{\times}\rightarrow A^{\times}$}, $a\mapsto a^{-1}$ is continuous at $1_A$. For a compact manifold $M$, the Fr\'echet algebra of smooth functions $C^{\infty}(M)$ is the prototype of such a continuous inverse algebra. We show that if $A$ is a complete commutative CIA, $G$ a compact group and $(A,G,\alpha)$ a dynamical system, then each character of the fixed point algebra $A^G$ can be extended to a character of $A$. In particular, the natural map on the level of the corresponding spectra $\Gamma_A\rightarrow\Gamma_{A^G}$, $\chi\mapsto\chi_{\mid A^G}$ is surjective.


\end{abstract}

\pagenumbering{arabic}

\thispagestyle{empty}

\section{Introduction}

The origin of this paper lies in the question whether there is a way to translate the geometric concept of a principal bundle to noncommutative geometry. In the case of vector bundles the Theorem of Serre and Swan (cf. \cite{Swa62}) gives the essential clue: The category of vector bundles over a compact space $X$ is equivalent to the category of finitely generated projective modules of $C(X)$. This observation leads to a notion of noncommutative vector bundles and is the connection between the topological K-theory based on vector bundles and the K-theory for C$^*$-algebras. For principal bundles, free and proper actions offer a good candidate for a notion of noncommutative principal bundles (see e.\,g.  \cite{BHS07, BDH13, Ellwood00, Ph09, SchWa15a, SchWa15b}). In a purely algebraic setting, the well-established theory of Hopf--Galois extensions provides a wider framework comprising coactions of Hopf algebras (e.\,g. \cite{Haj04,LaSu05,Sch04}). We also would like to mention the related notion of noncommutative principal torus bundles proposed by Echterhoff, Nest, and Oyono-Oyono \cite{ENOO09} (see also \cite{HaMa10}), which relies on a noncommutative version of Green's Theorem. A similar geometric approach based on transformation groups was developed in \cite{Wa11d, Wa11e}. In this context the question came up under which conditions a continuous character of a fixed point algebra of a dynamical system can be extended to a continuous character of the original algebra. A classical result for finite group actions can be found in \cite[Chapter 5, \S 2.1, Corollary 4]{Bou89}. Our results include the class of compact group actions on commutative CIAs, which are naturally encountered in K-theory and noncommutative geometry, usually as dense unital subalgebras of C$^*$-algebras. In fact, we show that if $A$ is a complete commutative CIA, $G$ a compact group and $(A,G,\alpha)$ a dynamical system, then each character of the fixed point algebra $A^G$ is continuous and can be extended to a continuous character of $A$.

\subsection*{Preliminaries and notations}
All algebras are assumed to be complex if not mentioned otherwise. If $A$ is an algebra, we write $\Gamma_A:=\Hom_{\text{alg}}(A,\mathbb{C})\backslash\{\bf{0}\}$ for the spectrum of $A$ (endowed with the topology of pointwise convergence on $A$). 
Moreover, given a dynamical system $(A,G,\alpha)$, we write $A^G:=\{a\in A:\,(\forall g\in G)\,\,\alpha(g).a=a\}$ for the corresponding fixed point algebra. Finally, we write $\widehat{G}$ for the set of all equivalence classes of finite-dimensional irreducible representations of a compact group $G$. 

\subsection*{Acknowledgments} 
We thank Henrik Sepp\"anen and Erhard Neher for useful discussions on this topic. 

\section{Extending characters on fixed point algebras}\label{extending characters on fixed point algebras}
In this section we state and prove our results on the extendability of characters on fixed point algebras. We first recall a useful result on the Fourier decomposition of dynamical systems with compact structure group. 

\begin{lemma}\emph{(}A structure theorem for dynamical systems\emph{)}.\label{str thm of G-mod}
Let $A$ be a complete unital locally convex algebra, $G$ a compact group and $(A,G,\alpha)$ a dynamical system. Further, given $\pi\in\widehat{G}$, let $\chi_{\pi}$ be the character of $\pi$ and $d_{\pi}:=\chi_{\pi}(1_G)$ the degree of $\pi$ \emph{(}i.e., the dimension of $\pi$\emph{)}. If
\[P_{\pi}(a):=d_{\pi}\cdot\int_G\overline{\chi_{\pi}}(g)\cdot(\alpha(g).a)\,dg,
\]where $a\in A$ and $dg$ denotes the normalized Haar measure on $G$, then the following assertions hold:

\begin{itemize}
\item[\emph{(a)}]
For each $\pi\in\widehat{G}$ the map $P_{\pi}:A\rightarrow A$ is a continuous $G$-equivariant projection onto the $G$-invariant subspace $A_{\pi}:= P_{\pi}(A)$. In particular, $A_{\pi}$ is algebraically and topologically a direct summand of $A$.

\item[\emph{(b)}]
The module direct sum $A_{\emph{\text{fin}}}:=\bigoplus_{\pi\in\widehat{G}}A_{\pi}$ is a dense subalgebra of $A$.
\end{itemize}
\end{lemma}

\begin{proof}
\,\,\,The assertions follow from \cite[Lemma 3.2 and Theorem 4.22]{HoMo06}
\end{proof}


\begin{proposition}\label{extension of char on fixed point algebras I} 
Let $A$ be a complete unital locally convex algebra, $G$ a compact group and $(A,G,\alpha)$ a dynamical system. Further, let $A^G$ be the corresponding fixed point algebra. Then the following assertions hold:
\begin{itemize}
\item[\emph{(a)}]
If $I$ is a proper left ideal in $A^G$, then 
\[A_{\emph{\text{fin}}}\cdot I=\bigoplus_{\pi\in\widehat{G}}A_{\pi}\cdot I
\]defines a proper left ideal in $A_{\emph{\text{fin}}}$ which contains $I$.
\item[\emph{(b)}]
If $I$ is a proper closed left ideal in $A^G$ and $J$ is the closure of $A_{\text{fin}}\cdot I$ in $A_{\text{fin}}$,
then $J$ is a proper closed left ideal in $A_{\emph{\text{fin}}}$ which contains $I$.
\end{itemize}
\end{proposition}

\begin{proof}
\,\,\,(a) We first observe that $A^G$ coincides with $A_{\bf 1}$ (where {\bf 1} stands for the equivalence class of the trivial representation). Hence $I\subseteq A^G$ is contained in $A_{\text{fin}}$ and thus $A_{\text{fin}}\cdot I$ is the left ideal of $A_{\text{fin}}$ generated by $I$. Using the integral formula for $P_{\pi}$ from Lemma 2.1, we see that $A_{\pi}\cdot I\subseteq A_{\pi}$, entailing that the sum in part (a) is direct. To see that $A_{\text{fin}}\cdot I$ is proper, we assume the contrary, i.e., that \[1_A\in A_{\text{fin}}\cdot I=\bigoplus_{\pi\in\widehat{G}}A_{\pi}\cdot I.
\]Then $1_A\in A^G$ implies that $1_A\in A^G\cdot I=I$, which contradicts the fact that $I$ is a proper left ideal of $A^G$. Thus, $A_{\text{fin}}\cdot I$ is a proper ideal in $A_{\text{fin}}$ which contains $I$.

(b) Part (a) of the lemma and the definition of $J$ imply that $J$ is a closed left ideal in $A_{\text{fin}}$ which contains $I$. To see that $J$ is proper, we again assume the contrary, i.e., that $1_A\in J$. Then there exists a net $(a_{\alpha})_{\alpha\in\Gamma}$ in $A_{\text{fin}}\cdot I$ such that $\lim_{\alpha}a_{\alpha}=1_A$. Therefore, the continuity of the projection $P_{\bf 1}:A\rightarrow A$ onto $A^G$ leads to
\[1_A=P_{\bf 1}(1_A)=P_{\bf 1}(\lim_{\alpha}a_{\alpha})=\lim_{\alpha}P_{\bf 1}(a_{\alpha}).
\]Since $I$ is closed in $A^G$ and $P_{\bf 1}(a_{\alpha})\in A^G\cdot I=I$ for all $\alpha\in\Gamma$, we conclude that $1_A\in I$. This contradicts the fact that $I$ is a proper ideal of $A^G$. Thus, $J$ is a proper closed left ideal in $A_{\text{fin}}$ which contains $I$.
\end{proof}

\begin{lemma}\label{extension of char on fixed point algebras II}
Let $A$ be a topological algebra and $A'$ a dense subalgebra of $A$. If $I$ is a proper closed left ideal in $A'$, then $\overline{I}$ is a proper closed left ideal in $\overline{A'}=A$.
\end{lemma}


\begin{proof}
\,\,\,A short calculation shows that $\overline{I}$ is a closed left ideal in $\overline{A'}=A$. Next, we note that $I=\overline{I}\cap A'$. Indeed, the inclusion $``\subseteq "$ is obvious and for the other inclusion we use the fact that $I$ is closed in $A'$. Thus, if $\overline{I}$ is not proper, i.e., $\overline{I}=A$, then $I=A'$ which contradicts the fact that $I$ is a proper ideal of $A'$. Hence, $\overline{I}$ is a proper closed left ideal in $A$.
\end{proof}

\begin{proposition}\label{extension of char on fixed point algebras III}\emph{(}Extending ideals\emph{)}.
Let $A$ be a complete unital locally convex algebra, $G$ a compact group and $(A,G,\alpha)$ a dynamical system. Then each proper closed left ideal in $A^G$ is contained in a proper closed left ideal in $A$.
\end{proposition}

\begin{proof}
\,\,\,If $I$ is a proper closed left ideal in $A^G$, then Proposition \ref{extension of char on fixed point algebras I} (b) implies that $I$ is contained in a proper closed left ideal in $A_{\text{fin}}$. Since $A_{\text{fin}}$ is a dense subalgebra of $A$ by Lemma \ref{str thm of G-mod} (b), the claim is a consequence of Lemma \ref{extension of char on fixed point algebras II}.
\end{proof}


\begin{theorem}\label{extension of char on fixed point algebras IV}\emph{(}Extending characters\emph{)}.
Let $A$ be a complete commutative CIA, $G$ a compact group and $(A,G,\alpha)$ a dynamical system. 
Then each character $\chi:A^G\rightarrow\mathbb{C}$ is continuous and extends to a continuous character $\widetilde{\chi}:A\rightarrow\mathbb{C}$.
\end{theorem}

\begin{proof}
\,\,\,We first note that $A^G$ carries the structure of a CIA in its own right (cf. \cite[Proposition 1.7 (a)]{Bi10}). Therefore, given a character $\chi$ on $A^G$, it follows from \cite[Lemma 2.3]{Bi10} that $\chi$ is continuous and thus in turn that the kernel $I:=\ker\chi$ is a proper closed ideal in $A^G$. Now, Proposition \ref{extension of char on fixed point algebras III} implies that $I$ is contained in a proper closed ideal in $A$. In particular, it is contained in a proper maximal ideal $J$ of $A$. According to \cite[Lemma 2.2.2]{Bi04}, $J$ is the kernel of some character $\widetilde{\chi}:A\rightarrow\mathbb{C}$, which is again continuous by \cite[Lemma 2.3]{Bi10}. Since $I$ is a maximal ideal in the unital algebra $A^G$ and 
\[I=I\cap A^G\subseteq J\cap A^G\subseteq A^G,
\]we conclude that $I=J\cap A^G$. Hence, $A^G=I\oplus\mathbb{C}=(J\cap A^G)\oplus\mathbb{C}$ proves that $\widetilde{\chi}$ extends $\chi$.
\end{proof}

\begin{remark}
From a topological point of view it is a natural ambition to try to extend a continuous character on a fixed point algebra of some dynamical system to a continuous character on the original algebra. Unfortunately, it is not clear at all if this works for an arbitrary complete commutative unital locally convex algebra, because maximal ideals need in generally not be closed. For example the set $C_{\text{c}}^{\infty}(M)$ of compactly supported smooth functions on some non-compact manifold $M$ is a proper ideal in $C^{\infty}(M)$ and therefore contained in a proper maximal ideal which cannot be closed since $C_{\text{c}}^{\infty}(M)$ is dense. Thus, the reason for considering commutative CIAs in the previous theorem comes from the fact that they provide a class of commutative unital algebras for which all characters are automatically continuous (cf. \cite[Lemma 2.3]{Bi10}). Anyway, in the more general situation of a complete commutative unital locally convex algebra $A$, a similar argument as in the proof of Theorem \ref{extension of char on fixed point algebras IV} shows that each continuous character $\chi:A^G\rightarrow\mathbb{C}$ can always be extended to a character $\widetilde{\chi}:A\rightarrow\mathbb{C}$.
\end{remark}

\begin{corollary}
Suppose we are in the situation of Theorem \ref{extension of char on fixed point algebras IV}. Then the natural map on the level of the corresponding spectra $\Gamma_A\rightarrow\Gamma_{A^G}$, $\chi\mapsto\chi_{\mid A^G}$ is surjective.
\end{corollary}

\begin{proof}
\,\,\,The claim is a direct consequence of Theorem \ref{extension of char on fixed point algebras IV}.
\end{proof}

\begin{corollary}\label{extension of char on fixed point algebras V}
Let $P$ be a compact manifold and $G$ be a compact group. If $(C^{\infty}(P),G,\alpha)$ is a dynamical system, then each character $\chi:C^{\infty}(P)^G\rightarrow\mathbb{C}$ extends to a character $\widetilde{\chi}:C^{\infty}(P)\rightarrow\mathbb{C}$.
\end{corollary}

\begin{proof}
\,\,\,This claim again directly follows from Theorem \ref{extension of char on fixed point algebras IV}, since $C^{\infty}(P)$ is a complete commutative CIA.
\end{proof}

\begin{remark}
(a) An application of Corollary \ref{extension of char on fixed point algebras V} can be found in \cite{Wa11d,Wa11e}. Indeed, there we show that the map
\[\Phi:P/G\rightarrow\Gamma_{C^{\infty}(P)^G},\,\,\,p.G\mapsto \delta_p
\]is a homeomorphism and Corollary \ref{extension of char on fixed point algebras V} is used to show its surjectivity (here, $\delta_p$ denotes the natural evaluation homomorphism in $p\in P$).
 
(b) Another application can be found in \cite{NeSe09}. There, \cite[Proposition 3.2]{NeSe09} is a special case of Theorem \ref{extension of char on fixed point algebras IV}.
\end{remark}

\vspace*{1cm}

\noindent
Stefan Wagner\\
Universit\"at Hamburg\\
Fachbereich Mathematik\\
Bereich Algebra und Zahlentheorie\\
Bundesstra\ss e 55 (Geomatikum)\\
20146 Hamburg, Germany \\
\url{stefan.wagner@uni-hamburg.de}

\end{document}